\documentclass[reqno,11pt,letterpaper]{article}
\usepackage{amsmath,amssymb,amsthm,graphicx,mathrsfs,url}
\usepackage[usenames,dvipsnames]{color}
\usepackage[colorlinks=true,linkcolor=Red,citecolor=Green,urlcolor=Blue]{hyperref}
\usepackage[super]{nth}
\usepackage[open, openlevel=2, depth=3, atend]{bookmark}
\hypersetup{pdfstartview=XYZ}

\def\?[#1]{\textbf{[#1]}\marginpar{\Large{\textbf{??}}}}

\numberwithin{equation}{section}

\newtheorem{theorem}{Theorem}[section]
\newtheorem{lem}[theorem]{Lemma}
\newtheorem{prop}[theorem]{Proposition}

\newcounter{conj}
\newtheorem{conjecture}[conj]{Conjecture}

\newcommand{\eps}{\epsilon}

\usepackage{comment}

\DeclareMathSymbol{\leqslant}{\mathalpha}{AMSa}{"36} 
\DeclareMathSymbol{\geqslant}{\mathalpha}{AMSa}{"3E} 
\DeclareMathSymbol{\eset}{\mathalpha}{AMSb}{"3F}     
\renewcommand{\leq}{\;\leqslant\;}                   
\renewcommand{\geq}{\;\geqslant\;}                   

\usepackage{bbm}

\newcommand{\E}{\mathbb{E}}
\newcommand{\cE}{\mathcal{E}}
\renewcommand{\P}{\mathbb{P}}

\def\eps{\varepsilon}

\def\bi{\begin{itemize}}
\def\ei{\end{itemize}}

\def\bnum{\begin{enumerate}}
\def\enum{\end{enumerate}}
\def\<#1{\langle #1 \rangle}

\def\u{$\mathbf{U_{n,g_n}}$}
\def\he{half-edge}

\def\core{\text{core}}
\def\vol{\text{vol}}
\def\cu{$\core(\mathbf{U_{n,g_n}})$}
\def\c*{$\core^*(\mathbf{U_{n,g_n}})$}
\def\cM{$\core^{<M}(\mathbf{U_{n,g_n}})$}
\def\d{\mathbf{d}}
\def\Ud{$\mathcal{U}(\d)$}
\def\dbr{doubly rooted tree}
\def\cmd{$\text{CM}(\d)$}
\def\kexp{$\kappa$-expander}
\def\dexp{$\delta$-expander}

\title{Large expanders in high genus unicellular maps}

\author{Baptiste Louf}

\begin{document}

%
 \maketitle
\begin{abstract}
We study large uniform random maps with one face whose genus grows linearly with the number of edges. They can be seen as a model of discrete hyperbolic geometry. In the past, several of these hyperbolic geometric features have been discovered, such as their local limit or their logarithmic diameter. In this work, we show that with high probability such a map contains a very large induced subgraph that is an expander.
\end{abstract}

\section{Introduction}
\paragraph{Combinatorial maps} Combinatorial maps are discrete geometric structures constructed by gluing polygons along their sides to form (compact, connected, oriented) surfaces. They appear in various contexts, from computer science to mathematical physics, and have been given a lot of attention in the past few decades. The first model that was extensively studied is \emph{planar maps} (or maps of the sphere), starting with their enumeration \cite{Tut62,Tut63} by generating function methods. Later on, explicit constructions put planar maps in bijection with models of decorated trees \cite{Sch98these,Sch98,BDG04,BF12,AP15}, and geometric properties of large random uniform planar maps have been studied \cite{AS03,CS04,LG11,Mie11}. All these works were later extended to maps on fixed surfaces of any genus (see for instance \cite{LW72,BC86} for enumeration, \cite{CMS09,Lep19} for bijections, and 
\cite{Bet16} for random maps). 
\paragraph{High genus maps}
Much more recently, another regime of maps has been studied: high genus maps, that is (sequences of) maps whose genus grows linearly in the size of the map. The main goal is to study the geometric properties of a random uniform such map as the size tends to infinity. By the Euler formula, the high genus implies that these maps have negative average discrete curvature. They must therefore have hyperbolic features, some of whose have been identified in previous works \cite{ACCR13,Ray13a,BL19,BL20, Lou20}.

Mostly, two types of models of high genus maps have been dealt with. First, \emph{unicellular maps}, i.e. maps with one face, who are easier to tackle thanks to an explicit bijection \cite{CFF13}, and then more general models of maps like triangulations or quadrangulations. It is believed that both models have a similar behaviour.

The local behaviour of high genus maps around their root is now well understood (\cite{ACCR13} in the unicellular case, \cite{BL19,BL20} in the general case), and some global properties have been tackled: the planarity radius \cite{Lou20} (see also \cite{Ray13a} for unicellular maps) and the diameter (\cite{Ray13a} for unicellular maps, still open for other models). 
\paragraph{Large expanders: a result and a conjecture}

In this paper we deal with yet another property: the presence of large expanders inside our map, in the case of unicellular maps. Contrary to the previous properties, this involves the whole geometric structure of the map. Expander graphs are very well connected graphs, in which every set of vertices has a large number of edges going out of it, which is a typical hyperbolic behaviour\footnote{indeed, for each set, a large proportion of its mass is contained on its boundary}. Unfortunately, it is impossible that the whole map itself is an expander, since it can be shown that finite but very large pending trees (which have very bad influence on the expansion of the graph) can be found somewhere in the map. However, it can be shown that most of the map is an expander, in the following sense.

Let $\frac{g_n}{n}\to \theta\in (0,1/2)$, and let \u{} be a uniform unicellular map of size $n$ and genus $g_n$.
\begin{theorem}\label{thm}
For all $\eps>0$, there exists a $\kappa>0$ depending only of $\eps$ and $\theta$ such that the following is true.
With high probability\footnote{throughout the paper, we will write \emph{with high probability} or \emph{whp} in lieu of \emph{with probability $1-o(1)$ as $n\to\infty$}.}, there exists an induced subgraph $G_n$ of \u{} that has at least $(1-\eps)n$ edges and is a $\kappa$-expander.
\end{theorem}

It is natural to conjecture that a similar results holds for more general models of maps (i.e., without a fixed number of faces). For instance, let $\mathbf{T_{n,g_n}}$ be a uniform triangulation of genus $g_n$ with $3n$ edges. The following conjecture (and the present work) comes from a question of Itai Benjamini (private communication) about large expanders in high genus triangulations.

\begin{conjecture}\label{conj}
For all $\eps>0$, there exists a $\kappa>0$ depending only of $\eps$ and $\theta$ such that the following is true.
With high probability, there exists an induced subgraph $G_n$ of $\mathbf{T_{n,g_n}}$ that has at least $(1-\eps)3n$ edges and is a $\kappa$-expander.
\end{conjecture}

This conjecture deals with the entire structure of the map, therefore we believe it is a very ambitious open problem about the geometry of high genus maps. Some other conjectures might be easier to tackle, see \cite{Lou20}.

\paragraph{Structure of the paper}

The proof of the main result involves the refinement/extension of several known results. We chose to push all the technical proofs to the appendix, in order to make it clear how we combine these results to obtain the proof of our main theorem. 

The main objects are defined in the next section, then we give an outline of the proof. The proof consists roughly of two halves: showing that the core is an expander (Section~\ref{sec_core_exp}) and showing that a well defined ``almost core" is still an expander while having a large proportion of the edges (Section~\ref{sec_almost}). Finally, in the appendix, we give the proofs of the technical lemmas, as well as a causal graph of all the parameters involved (see Section~\ref{sec_causal} for more details).

\paragraph{Acknowledgements} The author is grateful to Thomas Budzinski, Guillaume Chapuy, Svante Janson and Fiona Skerman for useful comments and discussions about this work, and to Guillaume Conchon--Kerjan for pointing out an error in a previous version of this paper. Work supported by the Knut and Alice Wallenberg foundation.

\section{Definitions}
We begin with some definitions about graphs. Note that here we allow graphs to have loops and multiple edges, such objects are also commonly called multigraphs. We will write $e(G)$ for the number of edges in a graph $G$. An \emph{induced subgraph} of a graph $G$ is a graph obtained from $G$ by deleting some of its vertices and all the edges incident to these vertices. Given a subset $X$ of vertices of $G$, we write $G[X]$ for the induced subgraph of $G$ obtained by deleting all vertices that do not belong to $X$.  A \emph{topological minor} of a graph $G$ is a graph obtained from $G$ by deleting some of its vertices, some of its edges, and by ``smoothing" some of its vertices of degree $2$ as depicted in Figure~\ref{fig_smoothing}.

\begin{figure}
\center
\includegraphics[scale=0.5]{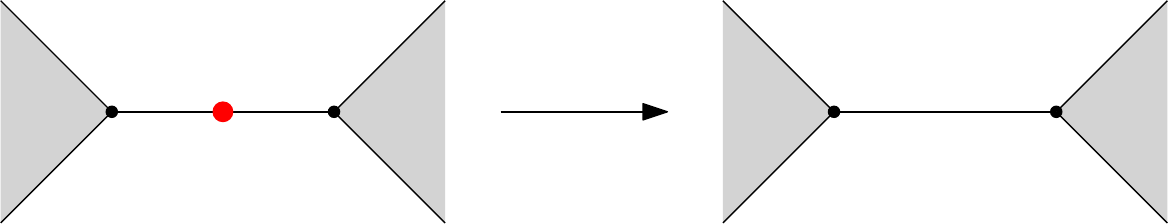}
\caption{Smoothing a vertex of degree $2$.}\label{fig_smoothing}
\end{figure}

Given a graph $G$ and a subset $X$ of its vertices, we define $\text{vol}(X)=\sum_{v\in X}\deg(v)$ and $\partial_G(X)$ as the number of edges of $G$ with exactly one endpoint in $X$. Then we set \[h_G(X)=\frac{\partial_G(X)}{\min(\text{vol}(X),\text{vol}(\overline X))},\] where $\overline X$ is the set of vertices of $G$ that do not belong to $X$. A graph $G$ is said to be a \emph{$\kappa$-expander} if the following inequality holds for every\footnote{it is easily verified that one only needs to check this inequality for subsets $X$ such that $G[X]$ is connected.} subset $X$ of vertices of $G$ such that $X\neq \emptyset$ and $\overline X\neq \emptyset$:
\[h_G(X)\geq \kappa.\]

A \emph{map} is the data of a collection of polygons whose sides were glued two by two to form a compact oriented surface. The interior of the polygons define the \emph{faces} of the map. After the gluing, the sides of the polygons become the \emph{edges} of the map, and the vertices of the polygons become the \emph{vertices} of the map. Alternatively, a map is the data of a graph endowed with a \emph{rotation system}, i.e. a clockwise ordering of \he{s} around each vertex. A \emph{unicellular map} of size $n$ is the data of a $2n$-gon whose sides were glued two by two to form a compact, connected, orientable surface. The genus $g$ of the surface is also called the \emph{genus} of the map. We will consider \emph{rooted maps}, i.e. maps with a distinguished oriented edge called the \emph{root}.
Let $\mathcal{U}_{n,g}$ be the set of rooted unicellular maps of size $n$ and genus $g$. A map of $\mathcal{U}_{n,g}$ has exactly $n+1-2g$ vertices by Euler's formula. We will denote by $\mathbf{U_{n,g}}$ a random uniform element of $\mathcal{U}_{n,g}$.

A \emph{tree} is a unicellular map of genus $0$. A \emph{doubly rooted tree} is a tree with a ordered pair of distinct marked vertices (but without a distinguished oriented edge). The \emph{size} of a doubly rooted tree is its number of edges. We set $dt_n$ to be the number of \dbr{s} of size $n$.
The \emph{core} of a unicellular map $m$, noted $\core(m)$, is the map obtained from $m$ by iteratively deleting all its leaves, then smoothing all its vertices of degree $2$ (see Figure~\ref{fig_core_decomp}). We do not make precise here how the root of $\core(m)$ is obtained from the root of $m$, we will only explain it in Section~\ref{sec_branches} (the only place where the root matters, for enumeration purposes, everywhere else in the paper we will only need the graph structure of the core).

\begin{figure}
\center
\includegraphics[scale=0.7]{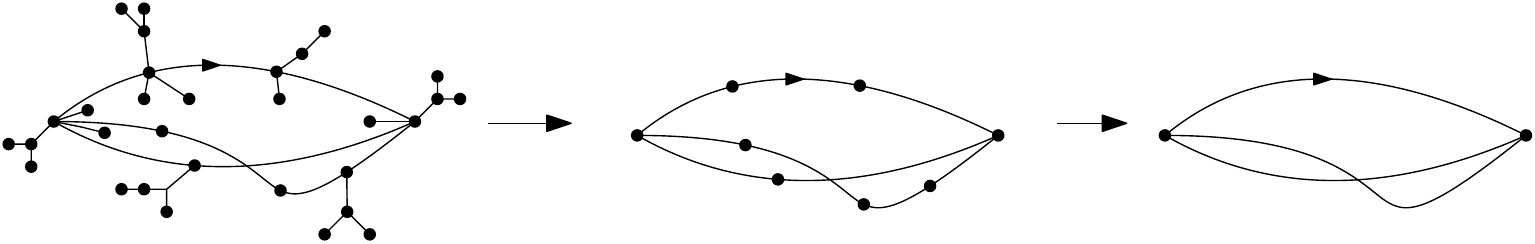}
\caption{The core decomposition of a unicellular map.}\label{fig_core_decomp}
\end{figure}

On the other hand, $m$ can be obtained from $\core(m)$ in a unique way by replacing each edge of $\core(m)$ by a doubly rooted tree. More precisely, given a \dbr{} $t$ and its two distinguished vertices $v_1$ and $v_2$, there is a unique simple path $p$ going from $v_1$ to $v_2$. Let $e_1$ (resp. $e_2$) be the edge of $p$ that is incident to $v_1$ (resp. $v_2$), and let $c_1$ (resp. $c_2$) that comes right before $e_1$ (resp. $e_2$) in the counterclockwise order around $v_1$ (resp. $v_2$). Now, we can remove an edge $e$ from $\core(m)$ to obtain a map with a pair of marked corners $c$ and $c'$, and we can glue $c_1$ on $c$ and $c_2$ on $c'$  (see Figure~\ref{fig_branches}). The set of the doubly rooted trees used to construct $m$ from $\core(m)$ will be called the branches of $m$. We also define $\core^{<M}(m)$ to be the map obtained from $m$ by replacing all its branches of size greater or equal to $M$ by a single edge. Notice that both $\core(m)$ and $\core^{<M}(m)$ are topological minors of $m$.

\begin{figure}
\center
\includegraphics[scale=0.7]{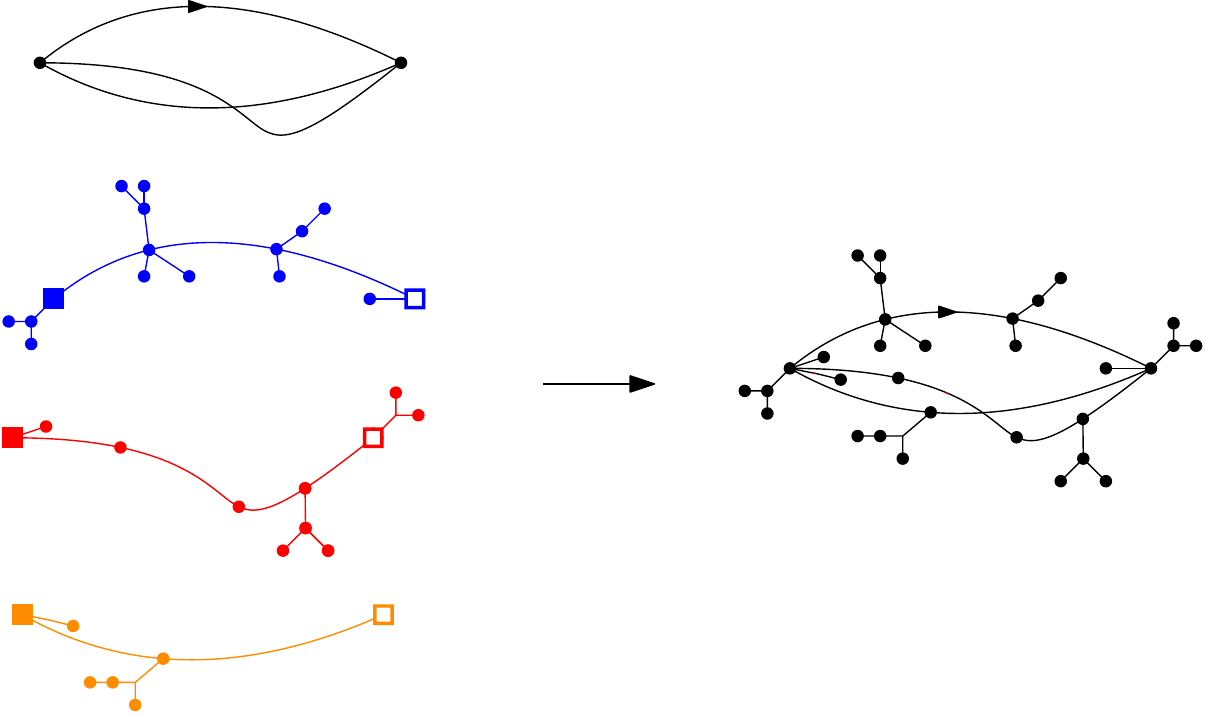}
\caption{Reconstructing a map from its core and branches.}\label{fig_branches}
\end{figure}

For any multiset of integers $\d=\{d_1,d_2,\ldots,d_k\}$, we set $|\d|=\sum_{i=1}^k d_i$. Let $\mathcal{U}(\d)$ be the set of rooted unicellular maps with vertex degrees given by $\d$\footnote{notice that if $|\d|$ is odd, or if $|\d|/2+k$ is even, then \Ud{} is empty.}. If $m$ is a unicellular map, let $\d(m)$ be the multiset of its vertex degrees. Now, we will define a random map $\text{CM}(\d)$ in the following way: take an arbitrary ordering $(d_1,d_2,\ldots,d_k)$ of $\d$, and let $v_1,v_2,\ldots,v_k$ be vertices such that $v_i$ has $d_i$ distinguishable dangling \he{s} arranged in  clockwise order around it. Now $\text{CM}(\d)$ is the random map obtained by taking a random uniform pairing of all the dangling \he{s}, and then picking a uniform oriented edge as the root (see Figure~\ref{fig_config}).

\begin{figure}
\center
\includegraphics[scale=0.5]{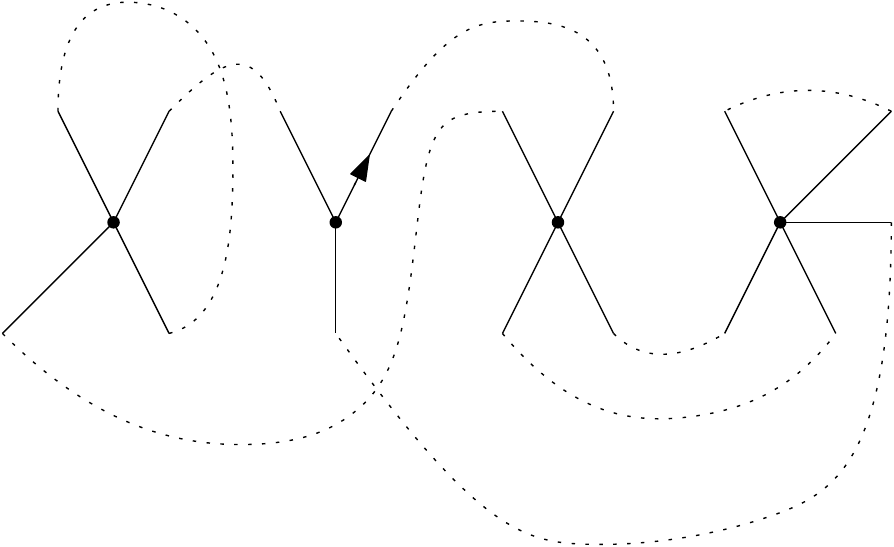}
\caption{The map configuration model for $\d=(3,4,4,5)$.}\label{fig_config}
\end{figure}

\section{Strategy of proof}
In this section, we give the outline of the proof of Theorem~\ref{thm}. First, the problem reduces to finding a good topological minor, because of the following theorem of Skerman and the author:
\begin{theorem}[\cite{LS20}, Theorem 1] \label{thm_topminor}
For all $\kappa, \alpha>0$, and for all $0<\alpha'<\alpha$, there exists a $\kappa'>0$ such that the following holds for every (multi)graph $G$.

If there exists a graph $H$ satisfying the following conditions:
\begin{itemize}
\item $e(H)\geq \alpha e(G)$,
\item $H$ is a topological minor of $G$,
\item $H$ is a $\kappa$-expander,
\end{itemize}
then there exists a graph $H^*$ satisfying the following conditions:
\begin{itemize}
\item $e(H^*)\geq \alpha' e(G)$,
\item $H^*$ is an induced subgraph of $G$,
\item $H^*$ is a $\kappa'$-expander.
\end{itemize}
\end{theorem}

Furthermore, we will prove the following:
\begin{prop}\label{prop_core_M}
For all $\eps>0$, there exist $M$ and $\kappa>0$ that depend only on $\theta$ and $\eps$ such that the following is true whp:
\begin{itemize}
\item \cM{} has more than $(1-\eps)n$ edges,
\item \cM{} is a $\kappa$-expander.
\end{itemize}
\end{prop}
It is clear that Theorem~\ref{thm} is an immediate corollary of Theorem~\ref{thm_topminor} and Proposition~\ref{prop_core_M}, since \cM{} is a topological minor of \u{}.

A key tool of the proof of Proposition~\ref{prop_core_M} is the following result.

\begin{prop}\label{prop_core}
There exists a \textbf{universal}\footnote{i.e. independent of $\theta$.} $\delta>0$ such that, whp, \cu{} is a $\delta$-expander.
\end{prop}
In the next section, we will prove Proposition~\ref{prop_core}, and Section~\ref{sec_almost} is devoted to the proof of Proposition~\ref{prop_core_M}.
\section{The core is an expander}\label{sec_core_exp}
Here, we will prove Proposition~\ref{prop_core} by comparing \cu{} to a well chosen map configuration model. We begin with two results about this model. We will consider sets $\d=(d_1,d_2,\ldots,d_k)$ that do not contain any $1$'s or $2$'s, with $|\d|=2n$, such that  $k+n$ is odd.

\begin{lem}\label{lem_config_unicellular}
The map \cmd{} is unicellular with probability greater than
\[\frac {1+o(1) }{3n}\]
as $n\to\infty$, where the $o(1)$ is independent of $\d$.
\end{lem}
The proof of this lemma is a little technical, it actually needs a refinement of an argument of \cite{BCP19}, it will be given in the appendix.
%
This next proposition states that \cmd{} is an expander with very high probability.

\begin{prop}\label{prop_config_expander}
The map \cmd{} is a $\delta$-expander (with the same $\delta$ as in Proposition~\ref{prop_core}) with probability
\[1-\P(\text{\cmd{} is disconnected})-o\left(\frac 1 n\right)\]
as $n\to\infty$, where the $o\left(\frac 1 n\right)$ is independent of $\d$.
\end{prop}
The proof is rather technical, but it is heavily inspired by \cite{HLW06,KW14}. A careful analysis of the cases is needed, but there is no original idea involved, hence we delay it to the appendix.

We are now ready to prove Proposition~\ref{prop_core}.
\begin{proof}[Proof of Proposition~\ref{prop_core}]
Let $\d=\d(\core(\mathbf{U_{n,g_n}}))$. The map \cu{} has genus $g_n\to\infty$ and only vertices of degree greater or equal to $3$, hence $|\d|\to \infty$.

Now, conditionally on $\d$, \cu{} is uniform in \Ud{}. Also, \cmd{} conditioned on having one face is uniform in \Ud{}. Hence, the probability of \cu{} not being a $\delta$-expander is
\[\P\left(\text{\cmd{} is not a $\delta$-expander}|\text{\cmd{} is unicellular} \right)\]
which we can upper bound by
\[\frac{\P\left(\text{\cmd{} is connected and not a $\delta$-expander}\right)}{\P(\text{\cmd{} is unicellular})}\]
(because all unicellular maps are connected).
This is $o(1)$ by Lemma~\ref{lem_config_unicellular} and Proposition~\ref{prop_config_expander}. This $o(1)$ does not depend on $\d$, hence the proof is finished.
\end{proof}
\section{Almost-core decomposition}\label{sec_almost}
In this section, we prove Proposition~\ref{prop_core_M}. Our strategy is the following: now that we know that \cu{} is an expander, we will add back to it the ``small" branches of \u{} to get very close to the size of \u{} without penalizing the expansion too much. We start with technical lemmas. The first one states that replacing edges by small doubly rooted trees does not change the Cheeger constant too much.
\begin{lem}\label{lem_adding_branches}
Let $H$ be a graph and $G$ be constructed by replacing each edge of $H$ by a doubly rooted tree of size $M$ or less. Then
\[h_G\geq \frac{h_H}{2M+1}.\]
\end{lem}
\begin{proof}
This proof is very similar to the proof of Lemma 5 of \cite{LS20}. 

In ${G}$, colour in red the vertices that come from $H$, and the rest in black. 
Let $Y$ be a subset of $V({G})$ such that  ${G}[Y]$ is connected (recall that we only need to consider connected subsets).  
Let $X$ be the set of red vertices in $Y$. We want to lower bound $h_{G}(Y)$ in terms of $h_H(X)$. See Figure~\ref{fig_go_big} for an illustration. 

If $X=\emptyset$, then $G[Y]$ is a tree on at most $ M-1$ vertices and so $\vol_G(Y)\leq 2(M-1)$ and $e_G(Y,\overline{Y})=2$. Hence 
\[h_{G}(Y)\geq \frac{e_G(Y,\overline{Y})}{\vol_G(Y)} \geq \frac{1}{M-1}.\]
Similarly if $\overline{X}=\emptyset$ then $h_G(Y)=h_G(\overline{Y})\geq 1/(M-1)$. 

Now, consider the case $X\neq\emptyset$ and $\overline{X}\neq \emptyset$. The number of edges of $H$ which are incident to a vertex of $X$ is $e_H(X)+e_H(X,\overline{X})\leq \vol_H(X)$. Each edge of $H$ is replaced by a tree with at most~$M$ edges, thus of volume at most $2(M-1)$. Therefore the total degree of the black vertices in $Y$ can be bounded above by~$2(M-1)\vol_H(X)$. Hence
\begin{equation}\label{e1}
    \vol_G(Y)\leq\vol_H(X)+2 (M-1) \vol_H(X)=(2M-1)\vol_H(X)
\end{equation}
and similarly
\begin{equation}\label{e2}
    \vol_G(\overline Y)\leq (2M-1)\vol_H(\overline X).
\end{equation}
Now, each edge counted in $e_H(X,\overline{X})$ corresponds to a \dbr{} in $G$ between $Y$ and
$\overline{Y}$, therefore
\begin{equation}\label{e3}
  e_{G}(Y,\overline{Y}) \geq e_H(X,\overline{X}).  
\end{equation}

Hence, by \eqref{e1}, \eqref{e2} and~\eqref{e3}:
\[h_G(Y) \geq \frac{1}{2M-1}h_H(X).\]
This concludes the proof.

\end{proof}

\begin{figure}
\center
\includegraphics[scale=0.5]{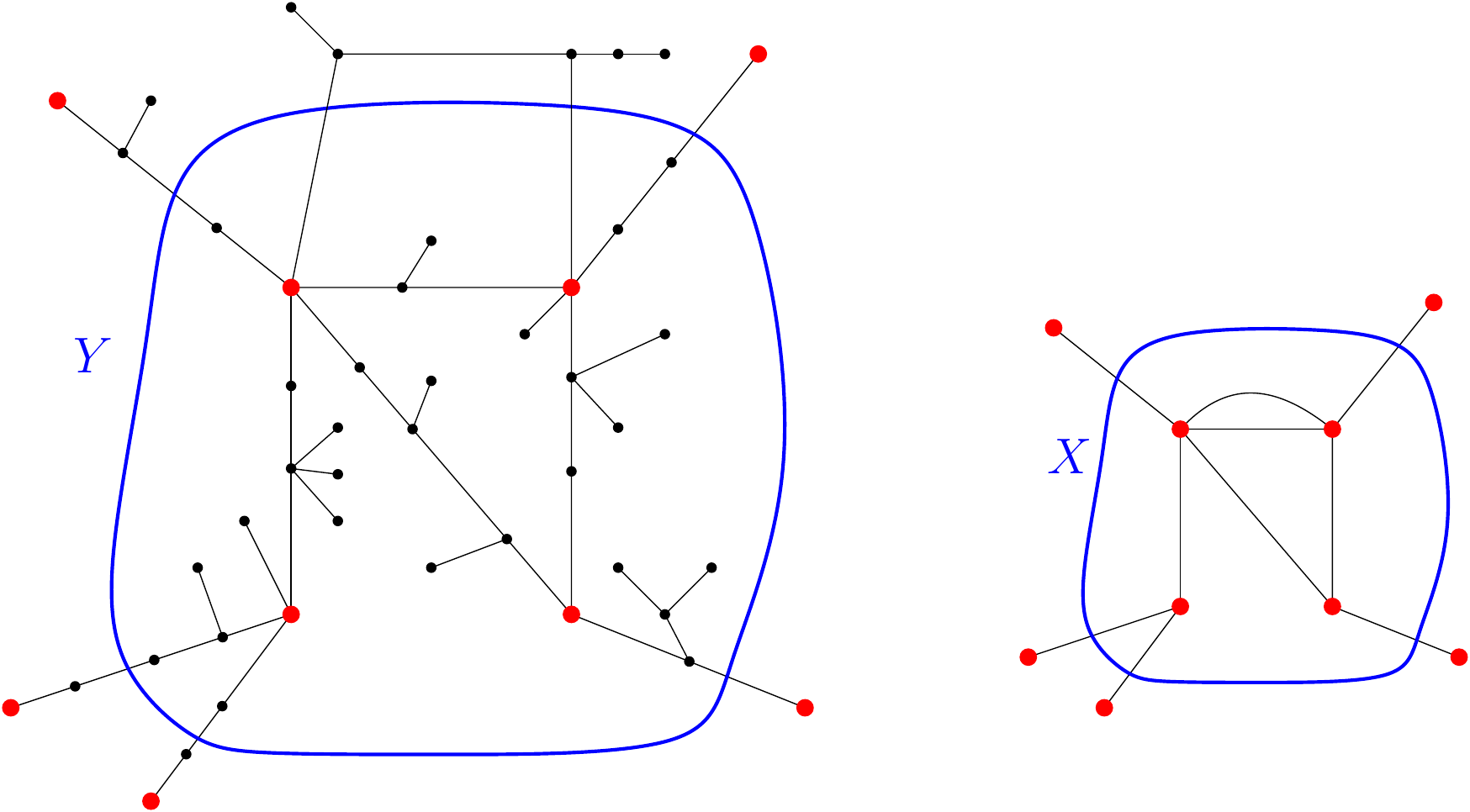}
\caption{Comparing the edge expansions of $Y$ in $G$ and $X$ in $H$. Here, $M=7$.}\label{fig_go_big}
\end{figure}

The next lemma states that the big branches of \u{} only make up for a very small proportion of its size.
\begin{lem}\label{lem_few_big_branches}
For all $\eps>0$, there exists a constant $M$ such that, whp, the total size of the branches of \u{} that are bigger than $M$ is less than $\eps n$.
\end{lem}
As for Proposition~\ref{prop_config_expander}, we postpone the proof to the appendix. We need a precise estimation of the second moment of the size of large branches, and it makes it a bit technical.

We are now ready to prove Proposition~\ref{prop_core_M}.

\begin{proof}[Proof of Proposition~\ref{prop_core_M}]
By Lemma~\ref{lem_few_big_branches}, we know that there exists an $M$ depending only on $\eps$ and $\theta$ such that \cM{} contains at least $(1-\eps)n$ edges whp. By Proposition~\ref{prop_core}, \cu{} is a \dexp{} whp. Now, \cM{} can be constructed out of \cu{} by replacing each edge by a \dbr{} of size $M$ or less, hence if we set $\kappa=\frac{\delta}{2M-1}$, then \cM{} is a \kexp{} whp.
\end{proof}
\appendix
\section{Proof of Lemma~\ref{lem_config_unicellular}}
Recall that we have a set $\d=(d_1,d_2,\ldots,d_k)$ that does not contain any $1$'s or $2$'s, with $|\d|=2n$, such that  $k+n$ is odd. We want to show that \cmd{} has only one face with probability $\Theta\left(\frac {1 }{n}\right)$. Our proof consists in estimating some quantities carefully in an argument of \cite{BCP19}. We however do not know of a more direct proof.

In \cite{BCP19}, the authors consider a model that is dual to ours, i.e. they glue polygons together, with the condition that there are few one-gons and digons (our case fits into their assumptions, since we have none). More precisely, they have as a parameter a list $\mathcal P_n$ of sizes of polygons that sum to $2n$ (this corresponds to our $\d$). The list $\mathcal P_n$ contains $\#\mathcal P_n$ elements (this corresponds to our $k$). An important parameter in their proofs is a random number $0\leq \tau_n\leq n$.

In section 4 of \cite{BCP19}, they control the number of vertices of their map, which is the number of faces in \cmd{}. More precisely, equation 13 writes the number of vertices as
\begin{equation}\label{eq_vertices}
X_{\tau_n}+V_{2(n-\tau_n)}.
\end{equation}
Let us define the notions used in~\eqref{eq_vertices}. First, if we condition on $\tau_n$, then both terms in~\eqref{eq_vertices} are independent. From now on, we condition on $\tau_n$.

It is shown in Section 4.3 that
\[X_{\tau_n}\stackrel{d_{\text{TV}}}{=}(1+o(1))\text{Poisson}\left(\log\left(\frac n {n-\#\mathcal P_n}\right)\right),\]
(note that this $o(1)$ can be made uniform in $\mathcal P_n$ by a classical diagonal argument). In particular, since we have no one-gons or digons, we have $\#\mathcal P_n\leq \frac 2 3 n$. This implies that
\begin{equation}\label{eq_p1}
\P(X_{\tau_n}=0)\geq 1/3 +o(1).
\end{equation}

Now let us turn to $V_{2(n-\tau_n)}$. For any $p$, $V_{2p}$ is the number of vertices in a uniform unicellular map on $p$ edges. We can calculate $\P(V_{2p}=1)$ for even $p$ :
\begin{equation}\label{eq_proba_one_vertex}
\P(V_{2p}=1)=\frac{\# \text{ of unicellular maps on $p$ edges with one vertex}}{\# \text{ of unicellular maps on $p$ edges}}.
\end{equation}

The denominator in the formula above is easy to enumerate, it is $(2p-1)!!$ (number of ways to pair the edges in a $2p$-gon). To enumerate the numerator, we will use \cite{GS98}, more precisely equation 14 for $x=1$, and then Corollary 4.2 for $g=p/2$. It is equal to 
\[\frac{(2p)!}{2^pp!(p+1)}.\]
Therefore, we have exactly 
\[P(V_{2p}=1)=\frac{1}{p+1}\]
if $p$ is even.
At the end of Section 4 in \cite{BCP19} (proof of Theorem 3), it is shown that $X_{\tau_n}+n-\tau_n+1$ has the same parity as $n+\#\mathcal P_n$, which corresponds to $n+k$ in our case (and we require it to be odd), therefore $n-\tau_n$ is even, if we condition on $X_{\tau_n}=0$.
Hence
\begin{equation}\label{eq_p2}
\P(V_{2(n-\tau_n)}=1|X_{\tau_n}=0)\geq \frac 1 {n-\tau_n+1}\geq \frac 1 {n+1}.
\end{equation}

We are ready to conclude the proof of Lemma~\ref{lem_config_unicellular}. 
Conditionally on $\tau_n$, by~\eqref{eq_p1} and ~\eqref{eq_p2}, the probability that \cmd{} has exactly one face is
\[\P(X_{\tau_n}=0)\P(V_{2(n-\tau_n)}=1|X_{\tau_n}=0)\geq\frac {1+o(1)}{3n}.\]
This quantity is independent of $\tau_n$ and uniform in $\d$, hence it finishes the proof of Lemma~\ref{lem_config_unicellular}.

\section{Proof of Proposition~\ref{prop_config_expander}}
We recall that we work with a set of vertex degrees $\d$ such that $|\d|=2n$ (the dependence of $\d$ in $n$ will be implicit). We want to prove that \cmd{} is a \dexp{} with very high probability, for some \textbf{universal} $\delta>0$. For the sake of simplicity, we will not make $\delta$ explicit, but we will show that it exists.

In what follows, we will consider $\d$ as an ordered list $(d_1,d_2,\ldots,d_k)$, and we will have a list of vertices $(v_1,v_2,\ldots,v_k)$ equipped with distinguishable dangling \he{s}, where $v_i$ has degree $d_i$. For all $I\subset[k]$, we set $\vol(I)=\sum_{i\in I} d_i$, and $N_V(\d)$ is the number of sets $I\in[k]$ such that $\vol(I)=V$. We will also write $v_I=\{v_i|i\in I\}$.

We begin by estimating $N_V(\d)$. All fractions are to be understood as their floor values, which we do not write to make the notation less cumbersome.
\begin{lem}\label{lem_estim_Nv}
Let $0<V\leq n$, then
\[N_V(\d)\leq V/3 \binom{2n/3}{V/3}.\]
\end{lem}

\begin{proof}
Since for all $i$ we have $d_i\geq 3$, we have $\vol(I)\geq 3|I|$, and thus if $I$ is such that $\vol(I)=V$, then $|I|\leq V/3$. Hence
\[N_V(\d)\leq \sum_{i=1}^{V/3} \binom{k}{i}.\]
But, since $d_i\geq 3$ for all $i$, we have $k\leq 2n/3$, thus
\[N_V(\d)\leq \sum_{i=1}^{V/3} \binom{2n/3}{i}.\]
Finally, since $V/3\leq \frac{1}{2}(2n/3)$, the sequence $\binom{2n/3}{i}$ is increasing in $i$ in the range $[1,V/3]$, therefore 
\[N_V(\d)\leq V/3 \binom{2n/3}{V/3}.\]
\end{proof}

Now, we will define a set of bad events. Let $\cE_V$ be the event that, in \cmd{}, there exists a set $I$ with $\vol(I)=V$ and such that among all the dangling \he{s} of $v_I$, strictly less than $\delta V$ get paired with dangling \he{s} of vertices outside $v_I$. Notice that \cmd{} is not a $\delta$-expander iff at least one of the $\cE_V$ happens. We will separate the analysis in two regimes, depending on the size of $V$. We introduce a universal, small enough $\eta>0$. We do not make it explicit, for the sake of simplicity, but we will show that it exists later on.
 
 \paragraph{Small subsets}
 
We will tackle the case $V\leq \eta n$. This proof follows the lines of~\cite{HLW06}[Theorem 4.16] in the case of regular graphs.
 
First we treat the case of $V=4$ or $V=6$. If we require that $\delta <\frac 1 6$, then if $\cE_4$ or $\cE_6$ happens, it implies that \cmd{} is disconnected. Hence
\begin{equation}\label{eq_A}
\P(\cE_4\cup\cE_6)\leq \P(\text{\cmd{} is disconnected}).
\end{equation}
From now on, $V\geq 8$.
 Given $I\subset [k]$ of volume $V$ and $H$ a subset of the \he{s} of $v_I$ we define the following event
 \[Y_{I,H}:=\text{all \he{s} of $H$ are matched along themselves.}\]
 If $H$ has cardinality $h$, then
\[\P\left(Y_{I,H}\right)=\frac{(h-1)!!(2n-h-1)!!}{(2n-1)!!}.\]
By a union bound and Lemma~\ref{lem_estim_Nv}, we have
\begin{align*}\label{eq_union_bound_small_set}
\P\left(\cE_V\right)&\leq N_V(\d) \sum_{(1-\delta)V<h\leq V}\binom{V}{h} \frac{(h-1)!!(2n-h-1)!!}{(2n-1)!!}\\
&\leq\sum_{(1-\delta)V<h\leq V}  V/3 \binom{2n/3}{V/3} \binom{V}{h} \frac{(h-1)!!(2n-h-1)!!}{(2n-1)!!}.
\end{align*}

By the classical inequality $\binom a b\leq \left(\frac {ae} b\right)^b$ and the fact that $h\leq V$, we obtain

\begin{align*}
\P\left(\cE_V\right)&
\leq \sum_{(1-\delta)V<h\leq V} V/3 \left(\frac{2ne/3}{V/3}\right)^{V/3} \left(\frac{V}{h}\right)^h \frac{h-1}{2n-1}\frac{h-3}{2n-3}\ldots \frac{1}{2n-h+1}\\
& \leq \sum_{(1-\delta)V<h\leq V} V/3 \left(\frac{2ne/3}{V/3}\right)^{V/3} \left(\frac{V}{h}\right)^h \left(\frac{h}{n}\right)^{h/2}\\
&\leq \sum_{(1-\delta)V<h\leq V} V/3 (2e)^{V/3}\left(\frac{n}{V}\right)^{V/3-h/2}.
\end{align*}
Now, for $\delta$ small enough\footnote{uniformly in $\d$.}, we have $V/3-h/2<-\frac{V}{7}$ when $(1-\delta)V<h\leq V$, hence, since there are less than $V$ terms in the sum above,  
\[\P\left(\cE_V\right)\leq V^2/3 (2e)^{V/3}\left(\frac{V}{n}\right)^{V/7}. \]
For $\eta>0$ small enough (and $n$ large enough), the RHS in the inequality above is decreasing in $V$ in the range $V\in[8,\eta n]$, and hence
\begin{equation}\label{eq_B}
\sum_{V=8}^{\eta n} \P\left(\cE_V\right)=o(1/n)
\end{equation}
uniformly in $\d$.

\paragraph{Big subsets}
We will now care about bigger bad subsets, this time we need to control probabilities more carefully. The following proof is adapted from \cite{KW14}[Section 7].

Given $0\leq y<u\leq 1$, let $X_{u,y}$ be the number of subsets $I\subset[k]$ of volume $un$ that have exactly $yn$ \he{s} that are paired with \he{s} outside $I$.

We have (using Lemma~\ref{lem_estim_Nv})
\begin{align*}
\E(X_{u,y})&=N_{un}(\d) \binom{un}{yn}\binom{2n-un}{yn} \frac{(yn)!(2n-un-yn-1)!!(un-yn-1)!!}{(2n-1)!!}\\
&\leq un/3 \binom{2n/3}{un/3}\binom{un}{yn}\binom{2n-un}{yn} \frac{(yn)!(2n-un-yn-1)!!(un-yn-1)!!}{(2n-1)!!}.
\end{align*}
Therefore, by Stirling's formula, we have
\begin{equation}\label{eq_log}
\log\E(X_{u,y})\leq n(f(u,y)+o(1)),
\end{equation}
where
\[f(u,y)=\log\left(\left(\frac{2^2}{u^u(2-u)^{2-u}}\right)^{1/3} \frac{u^u}{y^y (u-y)^{u-y}}\frac{(2-u)^{2-u}}{y^y(2-u-y)^{2-u-y}} \frac{y^y}{2}\left((2-u-y)^{2-u-y}(u-y)^{u-y}\right)^{1/2}\right).\]

Notice that 

\[f(u,0)=\frac{1}{6}\log(u^u(2-u)^{2-u})-\frac{1}{3}\log 2.\]
It is easily verified that this function is decreasing and tends to zero as $u\to 0$. Therefore, taking our $\eta>0$ from earlier, let $-c=f(\eta,0)/2$. We have, for all $u\geq \eta$, $f(u,0)\leq-2c$.
By continuity of $f(u,y)$, there exists $\delta>0$ small enough such that for all $u\geq \eta$ and $y<\eta\delta$ 
\[f(u,y)<-c.\]

Hence, by \eqref{eq_log}, the first moment method and a union bound, we have that (again, uniformly in $\d$):

\begin{equation}\label{eq_bound_big_sets}
\sum_{V\geq \eta n} \P(\cE_V)=o\left( \frac{1}{n}\right).
\end{equation}
\paragraph{Concluding the proof} Combining \eqref{eq_A}, \eqref{eq_B} and~\eqref{eq_bound_big_sets} yields the proof of Proposition~\ref{prop_config_expander}.


 \section{Proof of Lemma~\ref{lem_few_big_branches}}\label{sec_branches}
This section is devoted to the proof of Lemma~\ref{lem_few_big_branches}. The general idea of the proof is to approach the sizes of the branches in \u{} by random i.i.d. variables. This method is often called \emph{Poissonization} by abuse of language, and it relies on the saddle point method. We will directly apply the results of \cite{FlaSeg09}[Chapter VIII.8].

\paragraph{Counting doubly rooted trees}
We first need to compute the generating function of  doubly rooted trees counted by edges. Let 
\[D(z)=\sum_{s\geq 1} dt_s z^s,\]
and let
\[T(z)=\frac{1-\sqrt{1-4z}}{2z}-1\]
be the generating function of planar rooted trees with at least one edge, counted by edges.
Then, we can prove the following formula
\[D=T+TD\]
by considering the path between the two roots of a \dbr{} (see Figure~\ref{fig_decomp_dbr}).
\begin{figure}
\center
\includegraphics[scale=0.7]{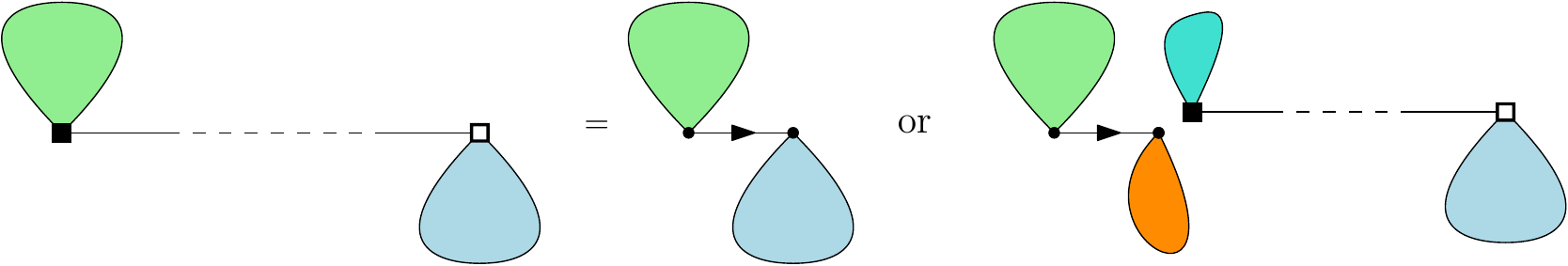}
\caption{Decomposing a \dbr{} along the path between its roots. The first (resp. second root) is a box (resp. square).}\label{fig_decomp_dbr}
\end{figure}
This directly implies
\begin{equation}\label{eq_gf_dbr}
D(z)=\frac{-2 z+1-\sqrt{1-4 z}}{4 z-1+\sqrt{1-4 z}}.
\end{equation}
Also, $C=\frac{z\partial}{\partial z}D$ is the series of \dbr{s} with a marked edge, and
\begin{equation}\label{eq_gf_dbr_marked}
C(z)=\frac{z(2-2 \sqrt{1-4 z}-4 z)}{(4 z-1+\sqrt{1-4 z})^{2} \sqrt{1-4 z}}.
\end{equation}

\paragraph{Studying the Poissonized law: defining the law}
For any $0<\beta<1/4$, we define the  two random variables $X_\beta$ and $Y_\beta$ with laws
\[\P(X_\beta=k)=\frac{([z^k]C(z))\beta^k}{C(\beta)}\]
and
\[\P(Y_\beta=k)=\frac{([z^k]D(z))\beta^k}{D(\beta)}.\]

Now, fix a constant $\theta\leq c\leq 1$ and $s_n\leq n-1$ such that $s_n\sim cn$.
We fix $\beta$ such that
\begin{equation}\label{eq_find_beta}
\E(X_\beta=k)+s_n\E(Y_\beta=k)=n.
\end{equation}
Let us first show that this $\beta$ exists. The equation above rewrites
\[(1+o(1))c\frac{C(\beta)}{D(\beta)}=1\]
Now, for $0<c\leq 1$, the equation $c\frac{C(\beta)}{D(\beta)}=1$ has a root in $[0,1/4)$, that is \[-\frac{c\left(\frac{c}{4}+\frac{\sqrt{c^{2}+8 c}}{4}\right)}{8}-\frac{c}{8}+\frac{1}{4}.\] Hence, by continuity, for $n$ large enough, \eqref{eq_find_beta} has a solution $\beta\in[0,1/4)$. Notice that $\beta$ is decreasing in $s_n$, hence there exists $\beta^*<1/4$ such that for all $s_n\geq \theta n$, we have $\beta\leq \beta^*$.

\paragraph{Studying the Poissonized law: depoissonization probability}
Now, the probability generating functions (in a variable $u$) for $X_\beta$ and $Y_\beta$ are respectively $\frac{C(\beta u)}{C(\beta)}$ and $\frac{D(\beta u)}{D(\beta)}$.
Let $Y_1$, $Y_2$,\ldots,$Y_{s_n}$ be i.i.d random variables distributed like $Y_\beta$, and let 
\[S=X_\beta+Y_1+Y_2+\ldots+Y_{s_n}.\]
We have $\E(S)=n$, therefore, by \cite{FlaSeg09}[Corollary VIII.3], we have
\begin{equation}\label{eq_proba_depoisson}
\P(S=n)=\Theta\left(\frac{1}{\sqrt{n}}\right)
\end{equation}
uniformly in $s_n\in [\theta n,n-1]$. 

\paragraph{Studying the Poissonized law: large deviations}
Next we will estimate the large deviations of the $Y_i$'s. Fix an $A>1$ such that $A\beta^*<1/4$ (recall that $1/4$ is the radius of convergence of $D$).
We have
\[\E(A^{Y_\beta})=\frac{D(A\beta)}{D(\beta)}\leq \frac{D(A\beta)}{A\beta}\frac \beta{D(\beta)}.\]
But $D(z)/z$ is a series with positive coefficients, and hence increasing. When $z\to 0$, $D(z)/z\to 1$ (the number of \dbr{s} with one edge). Therefore
\[\E(A^{Y_\beta})\leq \frac{D(A\beta^*)}{A\beta^*}=:W.\]
Hence, by the Markov inequality, 
\begin{align}\label{eq_proba_deviation}
P(Y_\beta\geq k)\leq \frac{W}{A^k}.
\end{align}

Now, let $Y^{(M)}_i=\mathbbm{1}_{Y_i>M}Y_i$ and
\[L^{(M)}=\mathbbm{1}_{X_\beta>M}X_\beta+\sum_{i=1}^{s_n}Y^{(M)}_i.\]
Let us also fix $1<B<A$, and set $r=B/A$, we have
\[\E\left(B^{Y^{(M)}_i}\right)=\sum_{k\geq M}P(Y_\beta\geq k)B^k\leq W\sum_{k\geq M}r^k=\frac{Wr^M}{1-r},\]
where the inequality follows from~\eqref{eq_proba_deviation}.
From now on, and until the end of the proof, let us fix $M$ large enough such that $\frac{Wr^M}{1-r}$ is small enough to guarantee $\E(B^{Y^{(M)}_i})^c\leq B^\eps/2$ for all $c\in[\theta, 1]$ in the inequality above. Note that this $M$ is independent of $s_n$ as long as $s_n\in[\theta n,n-1]$.

We can also show $\E(B^{\mathbbm{1}_{X_\beta>M}X_\beta})=O(1)$, hence by the Markov inequality we obtain
\begin{equation}\label{eq_deviation}
\P(L^{(M)}>\eps n)\leq O(1)\frac{\E\left(B^{Y^{(M)}_i}\right)^{s_n}}{B^{\eps n}}\leq 2^{-(1+o(1))n}
\end{equation}
uniformly in $s_n$.

Combining \eqref{eq_deviation} with \eqref{eq_proba_depoisson}, we obtain the following for $s_n\in[\theta n,n-1]$ and $n$ large enough 
\begin{equation}\label{eq_proba_large_deviation}
\P(L^{(M)}>\eps n|S=n) \leq \sqrt{2}^{-n}.
\end{equation}

\paragraph{Sizes of the branches} To go from a map to its core, one removes each branch and replaces it by an edge. Since a rooted map has no automorphisms, these branches can be put in a list of \dbr{} whose total size is the number of edges in the map. One actually needs to be a little more careful for the branch containing the root: mark the (unoriented) edge of the root in the \dbr{}, and replace it by a root edge with a coherent orientation (see Figure~\ref{fig_core_root}). This operation is bijective, therefore the list of sizes of the branches of \u{}, conditionally on \cu{} having $s_n+1$ edges is exactly the list of variables $(X_\beta,Y_1,Y_2,\ldots,Y_{s_n})$, conditionally on $S=n$.

\begin{figure}
\center
\includegraphics[scale=0.7]{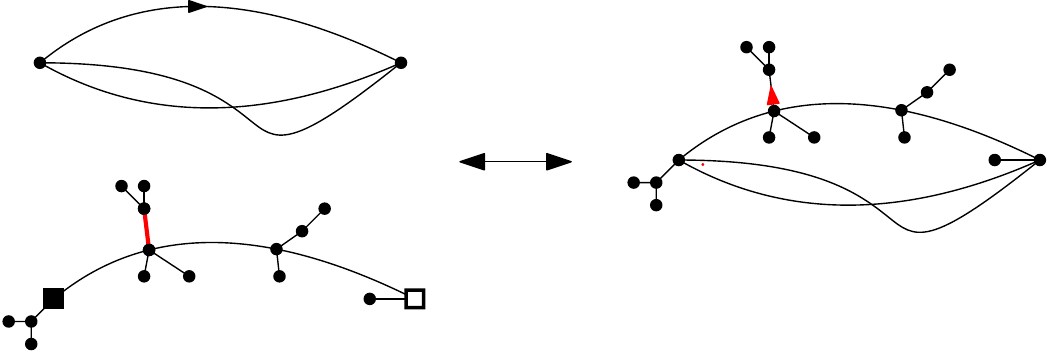}
\caption{The core decomposition for the branch that contains the root. An edge is marked, and the root is obtained by orienting this marked edge towards the second root in the clockwise exploration order around the branch.}\label{fig_core_root}
\end{figure}

We are ready to prove Lemma~\ref{lem_few_big_branches}.

\begin{proof}[Proof of Lemma~\ref{lem_few_big_branches}]
Since \cu{} is of genus $g_n$ and has one face, by the Euler formula, it has $s_n+1$ edges, with $s_n\geq 2g_n-1>\theta n$ (for $n$ large enough). The number of edges of \cM{} is exactly $n-L^{(M)}$ (conditionally on $S=n$).
We can apply~\eqref{eq_proba_large_deviation} conditionally on the number of edges of \cu{}, and it uniformly gives 
\[\P(\text{\cM{} has less than $(1-\eps)n$ edges})\leq \sqrt{2}^{-n}\]
which finishes the proof.
\end{proof}

\section{Causal graph of the parameters}\label{sec_causal}

In this paper, we introduce several parameters, and it might not be clear that there is no ``circularity" (especially since some of them are defined implicitly). The following figure presents a causal graph of the parameters and constants that appear in this problem.

\begin{figure}[!h]
\center
\includegraphics[scale=1]{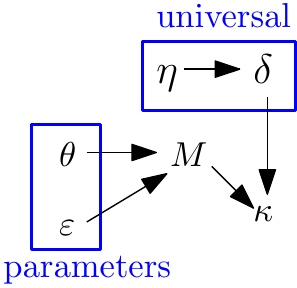}
\caption{Causal graph of the constants of this paper.}
\end{figure}
\footnotesize



\normalsize




\bibliographystyle{abbrv}

\bibliography{bibli}

\begin{thebibliography}{10}

\bibitem{AP15}
M.~Albenque and D.~Poulalhon.
\newblock Generic method for bijections between blossoming trees and planar
  maps.
\newblock {\em Electron. J. Comb. vol.22, paper P2.38}, 2015.

\bibitem{ACCR13}
O.~Angel, G.~Chapuy, N.~Curien, and G.~Ray.
\newblock The local limit of unicellular maps in high genus.
\newblock {\em Electron. Commun. Probab.}, 18(86):1--8, 2013.

\bibitem{AS03}
O.~Angel and O.~Schramm.
\newblock Uniform infinite planar triangulations.
\newblock {\em Comm. Math. Phys.}, 241(2-3):191--213, 2003.

\bibitem{BC86}
E.~A. Bender and E.~Canfield.
\newblock The asymptotic number of rooted maps on a surface.
\newblock {\em Journal of Combinatorial Theory, Series A}, 43(2):244 -- 257,
  1986.

\bibitem{BF12}
O.~Bernardi and E.~Fusy.
\newblock A bijection for triangulations, quadrangulations, pentagulations,
  etc.
\newblock {\em Journal of Combinatorial Theory, Series A 119, 1, 218-244},
  2012.

\bibitem{Bet16}
J.~Bettinelli.
\newblock Geodesics in {B}rownian surfaces ({B}rownian maps).
\newblock {\em Ann. Inst. Henri Poincar\'e Probab. Stat.}, 52(2):612--646,
  2016.

\bibitem{BDG04}
J.~Bouttier, P.~Di~Francesco, and E.~Guitter.
\newblock Planar maps as labeled mobiles.
\newblock {\em Elec. Jour. of Combinatorics Vol 11 R69}, 2004.

\bibitem{BCP19}
T.~Budzinski, N.~Curien, and B.~Petri.
\newblock Universality for random surfaces in unconstrained genus.
\newblock {\em Electron. J. Combin.}, 26(4):Paper No. 4.2, 34, 2019.

\bibitem{BL20}
T.~Budzinski and B.~Louf.
\newblock Local limits of bipartite maps with prescribed face degrees in high
  genus, 2020.

\bibitem{BL19}
T.~Budzinski and B.~Louf.
\newblock Local limits of uniform triangulations in high genus.
\newblock {\em Invent. Math.}, 223(1):1--47, 2021.

\bibitem{CFF13}
G.~Chapuy, V.~F\'eray, and E.~Fusy.
\newblock A simple model of trees for unicellular maps.
\newblock {\em Journal of Combinatorial Theory, Series A 120, 8, Pages
  2064-2092}, 2013.

\bibitem{CMS09}
G.~Chapuy, M.~Marcus, and G.~Schaeffer.
\newblock A bijection for rooted maps on orientable surfaces.
\newblock {\em SIAM J. Discrete Math.}, 23(3):1587--1611, 2009.

\bibitem{CS04}
P.~Chassaing and G.~Schaeffer.
\newblock Random planar lattices and integrated super{B}rownian excursion.
\newblock {\em Probab. Theory Related Fields}, 128(2):161--212, 2004.

\bibitem{FlaSeg09}
P.~Flajolet and R.~Sedgewick.
\newblock {\em Analytic Combinatorics}.
\newblock CUP, 2009.

\bibitem{GS98}
A.~Goupil and G.~Schaeffer.
\newblock Factoring {$n$}-cycles and counting maps of given genus.
\newblock {\em European J. Combin.}, 19(7):819--834, 1998.

\bibitem{HLW06}
S.~Hoory, N.~Linial, and A.~Wigderson.
\newblock Expander graphs and their applications.
\newblock {\em Bull. Amer. Math. Soc. (N.S.)}, 43(4):439--561, 2006.

\bibitem{KW14}
B.~Kolesnik and N.~Wormald.
\newblock Lower bounds for the isoperimetric numbers of random regular graphs.
\newblock {\em SIAM J. Discrete Math.}, 28(1):553--575, 2014.

\bibitem{LG11}
J.-F. Le~Gall.
\newblock Uniqueness and universality of the {B}rownian map.
\newblock {\em Ann. Probab.}, 41:2880--2960, 2013.

\bibitem{Lep19}
M.~Lepoutre.
\newblock Blossoming bijection for higher-genus maps.
\newblock {\em Journal of Combinatorial Theory, Series A}, 165:187 -- 224,
  2019.

\bibitem{Lou20}
B.~Louf.
\newblock Planarity and non-separating cycles in uniform high genus
  quadrangulations, 2020.

\bibitem{LS20}
B.~Louf and F.~Skerman.
\newblock Finding large expanders in graphs : from topological minors to
  induced subgraphs, in preparation.

\bibitem{Mie11}
G.~Miermont.
\newblock The {B}rownian map is the scaling limit of uniform random plane
  quadrangulations.
\newblock {\em Acta Math.}, 210(2):319--401, 2013.

\bibitem{Ray13a}
G.~Ray.
\newblock Large unicellular maps in high genus.
\newblock {\em Ann. Inst. H. Poincaré Probab. Statist.}, 51(4):1432--1456, 11
  2015.

\bibitem{Sch98these}
G.~Schaeffer.
\newblock {\em Conjugaison d'arbres et cartes combinatoires al\'eatoires}.
\newblock Th\`ese de doctorat, Universit\'e Bordeaux I, 1998.

\bibitem{Sch98}
G.~Schaeffer.
\newblock {\em Conjugaison d'arbres et cartes combinatoires al{\'e}atoires.
  {PhD} thesis}.
\newblock PhD thesis, 1998.

\bibitem{Tut62}
W.~T. Tutte.
\newblock A census of planar triangulations.
\newblock {\em Canad. J. Math.}, 14:21--38, 1962.

\bibitem{Tut63}
W.~T. Tutte.
\newblock A census of planar maps.
\newblock {\em Canad. J. Math.}, 15:249--271, 1963.

\bibitem{LW72}
T.~Walsh and A.~Lehman.
\newblock Counting rooted maps by genus. i.
\newblock {\em Journal of Combinatorial Theory, Series B}, 13(3):192--218,
  1972.

\end{thebibliography}

 \end{document}